\documentclass[12pt]{amsart}
\usepackage{verbatim,amssymb,latexsym}
\usepackage[all]{xy}
\usepackage{color}
\usepackage{amscd}

\usepackage[colorlinks,linkcolor=blue,citecolor=blue,urlcolor=red]{hyperref}

\newcommand{\cone}{\operatorname{cone}}

\renewcommand{\P}{\mathbf{P}}
\newcommand{\Q}{\mathbf{Q}}

\newcommand{\Z}{\mathbf{Z}}
\newcommand{\sA}{\mathcal{A}}

\newcommand{\sF}{\mathcal{F}}
\newcommand{\sG}{\mathcal{G}}
\newcommand{\sH}{\mathcal{H}}
\newcommand{\sI}{\mathcal{I}}

\newcommand{\sK}{\mathcal{K}}

\newcommand{\sX}{\mathcal{X}}

\newcommand{\bG}{\mathbb{G}}

\renewcommand{\phi}{\varphi}
\renewcommand{\epsilon}{\varepsilon}
\newcommand{\Spec}{\operatorname{Spec}}
\newcommand{\Ker}{\operatorname{Ker}}
\newcommand{\Coker}{\operatorname{Coker}}
\newcommand{\IM}{\operatorname{Im}}

\newcommand{\tors}{{\operatorname{tors}}}

\newcommand{\DM}{\operatorname{\mathbf{DM}}}
\newcommand{\Br}{\operatorname{Br}}

\newcommand{\Hom}{\operatorname{Hom}}
\newcommand{\End}{\operatorname{End}}

\newcommand{\Pic}{\operatorname{Pic}}
\newcommand{\NS}{\operatorname{NS}}
\newcommand{\Griff}{\operatorname{Griff}}

\newcommand{\Nis}{{\operatorname{Nis}}}

\newcommand{\uHom}{\operatorname{\underline{Hom}}}

\newcommand{\Ext}{\operatorname{Ext}}

\newcommand{\Tor}{\operatorname{Tor}}

\newcommand{\Ab}{\operatorname{\mathbf{Ab}}}

\newcommand{\car}{\operatorname{char}}

\newcommand{\alg}{{\operatorname{alg}}}
\newcommand{\ind}{{\operatorname{ind}}}
\newcommand{\num}{{\operatorname{num}}}

\newcommand{\nr}{{\operatorname{nr}}}

\newcommand{\Sm}{\operatorname{\mathbf{Sm}}}

\newcommand{\NST}{{\operatorname{\mathbf{NST}}}}
\newcommand{\EST}{{\operatorname{\mathbf{EST}}}}
\newcommand{\ES}{{\operatorname{\mathbf{ES}}}}
\newcommand{\HI}{\operatorname{\mathbf{HI}}}

\newcommand{\et}{{\text{\rm \'et}}}
\renewcommand{\o}{{\text{\rm o}}}

\newcommand{\eff}{{\text{\rm eff}}}

\newcommand{\by}{\xrightarrow}
\newcommand{\yb}{\xleftarrow}
\newcommand{\iso}{\by{\sim}}
\newcommand{\osi}{\yb{\sim}}
\newcommand{\inj}{\hookrightarrow}
\newcommand{\Inj}{\lhook\joinrel\longrightarrow}
\newcommand{\surj}{\rightarrow\!\!\!\!\!\rightarrow}
 
\newcommand{\colim}{\varinjlim}
\renewcommand{\lim}{\varprojlim}

\renewcommand{\qed}{\hfill $\Box$\medskip}

\newcounter{spec}
\newenvironment{thlist}{\begin{list}{\rm{(\roman{spec})}}%
{\usecounter{spec}\labelwidth=20pt\itemindent=0pt\labelsep=10pt}}%
{\end{list}}%

\newtheorem{Th}{Theorem}

\swapnumbers
\newtheorem{thm}{Theorem}[section]
\newtheorem{lemma}[thm]{Lemma}
\newtheorem{prop}[thm]{Proposition}

\theoremstyle{definition}

\newtheorem{rks}[thm]{Remarks}

\numberwithin{equation}{section}

\begin{document}
\title{The derived functors of unramified cohomology}
\author{Bruno Kahn}
\address{Institut de Math\'ematiques de Jussieu\\ Case 247\\ 4 Place Jussieu\\75252 Paris Cedex 05\\France}
\email{bruno.kahn@imj-prg.fr}
\author{R. Sujatha}
\address{Department of Mathematics, University of British Columbia\\ Vancouver\\Canada
V6T1Z2}
\email{sujatha@math.ubc.ca}
\date{September 15, 2017}
\begin{abstract} We study the first ``derived functors of unramified cohomology" in the sense of \cite{birat-tri}, applied to the sheaves $\bG_m$ and $\sK_2$. We find interesting connections with classical cycle-theoretic invariants of smooth projective varieties, involving notably a version of the Griffiths group and the group of  indecomposable $(2,1)$-cycles.
\end{abstract}
\thanks{The first author acknowledges the support of Agence Nationale de la Recherche (ANR) under reference ANR-12-BL01-0005 and the second author  that of NSERC Grant 402071/2011.} 
\maketitle

\tableofcontents

\section*{Introduction} To a perfect field $F$, Voevodsky associates in \cite{voetri} a \emph{triangulated category of (bounded above) effective motivic complexes} $\DM_-^\eff(F)=\DM_-^\eff$. In \cite{birat-tri}, we rather worked with the unbounded version $\DM^\eff$. We introduced a \emph{triangulated category of birational motivic complexes} $\DM^\o$, and constructed a triple of adjoint functors
\[\DM^\eff\begin{smallmatrix}\overset{R_\nr}{\longrightarrow}\\\overset{i^\o}{\longleftarrow}\\\overset{\nu_{\le 0}}{\longrightarrow} \end{smallmatrix}\DM^\o\]
with $i^\o$ fully faithful. Via $i^\o$, the homotopy $t$-structure of $\DM^\eff$ induces a $t$-structure on $\DM^\o$ (also called the homotopy $t$-structure), and the functors $\nu_{\le 0}$, $i^\o$ and $R_\nr$ are respectively right exact, exact and left exact.

The heart of $\DM^\eff$ is the abelian category $\HI$ of \emph{homotopy invariant Nisnevich sheaves with transfers} (see \cite{voetri}). The heart of $\DM^\o$ is the thick subcategory $\HI^\o\subset \HI$ of \emph{birational sheaves}: an object $\sF\in \HI$ lies in $\HI^\o$ if and only if it is locally constant for the Zariski topology.

In \cite{birat-tri} we also started to study the right adjoint $R_\nr$. Let 
$R^0_\nr=\sH^0\circ R_\nr:\HI\to \HI^\o$ be the induced functor. We proved that $R^0_\nr$ is given by the formula $R^0_\nr\sF=\sF_\nr$, where for a homotopy invariant sheaf $\sF\in \HI$, $\sF_\nr$ is defined by
\begin{equation}\label{eq1} \tag{1}
\sF_\nr(X) =\Ker\left(\sF(K)\to \prod_v \sF_{-1}(F(v))\right).
\end{equation}

Here $X$ is a smooth connected $F$-variety, $K$ is its function field, $v$ runs through all divisorial discrete valuations on $K$ trivial on $F$, with residue field $F(v)$, and $\sF_{-1}$ denotes the contraction of $\sF$ (see \cite{voepre} or \cite[Lect. 23]{mvw}). Thus $R^0_\nr \sF$ is the \emph{unramified part} of $\sF$.

Here is the example which connects the above to the classical situation of unramified cohomology. Let $i\ge 0$, $n\in\Z$ and let $m$ be an integer invertible in $F$. Then the Nisnevich sheaf $\sF=\sH^i_\et(\mu_m^{\otimes n})$ associated to the presheaf 
\[U\mapsto H^i_\et(U,\mu_m^{\otimes n})\]
defines an object of $\HI$, and $R^0_\nr\sF$ is the usual unramified cohomology \cite{ct}.

But the functor $R_\nr$ contains more information: for a general sheaf $\sF\in \HI$, the birational sheaves
\[R^q_\nr \sF =\sH^q(R_\nr\sF)\]
need not be $0$ for $q>0$. Can we compute them, at least in some cases?

In this paper, we try our hand at the simplest examples: $\sF=\bG_m(=\sK_1)$ and $\sF=\sK_2$. We cannot compute explicitly further than $q=2$, except for varieties of dimension $\le 2$; but this already yields interesting connections with other birational invariants. For simplicity, \emph{we restrict to the case where $F$ is algebraically closed}; throughout this paper, the cohomology we use is Nisnevich cohomology. The main results are the following:

\begin{Th}\label{t2.1} Let $X$ be a connected smooth projective $F$-variety. Then
\begin{thlist}
\item $R^0_\nr\bG_m(X) =F^*$.
\item $R^1_\nr\bG_m(X)\iso \Pic^\tau(X)$.
\item There is a short exact sequence
\[0\to D^1(X)\to R^2_\nr\bG_m(X)\to \Hom(\Griff_1(X),\Z)\to 0.\] 
\item For $q\ge 3$, we have short exact sequences
\begin{multline}\label{eq4.5}
0 \to \Ext_\Z(\NS_1(X,q-3),\Z) \to R_\nr^q\bG_m(X)\\
 \to \Hom_\Z(\NS_1(X,q-2),\Z) \to 0.
\end{multline}
\end{thlist}
Here the notation is as follows: $\Pic^\tau(X)$ is the group of cycle classes in $\Pic(X)=CH^1(X)$ which are numerically equivalent to $0$. We write $\Griff_1(X)=\Ker\left(A_1^\alg(X)\to N_1(X)\right)$, where $A_1^\alg(X)$ (resp. $N_1(X)$) denotes the group of $1$-cycles on $X$ modulo algebraic (resp. numerical) equivalence, and 
\[D^1(X)=\Coker\left(N^1(X)\to \Hom(N_1(X),\Z)\right)\]
where $N^1(X)=\Pic(X)/\Pic^\tau(X)$ and the map is induced by the intersection pairing. Finally, the groups $\NS_1(X,r)$ are those defined by Ayoub and Barbieri-Viale in \cite[3.25]{abv}.
\end{Th}

Note that $D^1(X)$ is a finite group since  $N_1(X)$ is finitely generated.

After Colliot-Th\'el\`ene complained that there was no unramified Brauer group in sight, we tried to invoke it by considering
\[\bG_m^\et = R\alpha_*\alpha^*\bG_m\]
where $\alpha$ is the projection of the \'etale site on smooth $k$-varieties onto the corresponding Nisnevich site. There is a natural map $\bG_m\to \bG_m^\et$, and

\begin{Th}\label{t2.1bis} The map $R^q_\nr \bG_m\to R^q_\nr \bG_m^\et$ is an isomorphism for $q\le 1$, and for $q=2$ there is an exact sequence for any smooth projective $X$:
\[0\to R^2_\nr \bG_m(X)\to R^2_\nr \bG_m^\et(X)\to \Br(X). \]
\end{Th}

Considering now $\sF=\sK_2$:

\begin{Th}\label{t3.1} We have an exact sequence
\begin{multline*}
0\to \Pic^\tau(X)F^*\to R^1_\nr\sK_2(X)\to H^1_\ind(X,\sK_2)\\
\by{\bar \delta} \Hom(\Griff_1(X),F^*)
\to R^2_\nr\sK_2(X)\to CH^2(X)
\end{multline*}
for any smooth projective variety $X$. Here 
\[H^1_\ind(X,\sK_2)=\Coker\left(\Pic(X)\otimes F^*\to H^1(X,\sK_2)\right)\]
 and 
 \[\Pic^\tau(X)F^*=\IM\left(\Pic^\tau(X)\otimes F^*\to H^1(X,\sK_2)\right).\]
\end{Th}

The group $H^1(X,\sK_2)$ appears in other guises, as the higher Chow group $CH^2(X,1)$ or as the motivic cohomology group $H^3(X,\Z(2))$; its quotient $H^1_\ind(X,\sK_2)$ has been much studied and is known to be often nonzero. Note that, while it is not clear from the literature whether there exist smooth projective varieties $X$ such that $\Hom(\Griff_1(X),\Z)\allowbreak\ne 0$, no such issue arises for $\Hom(\Griff_1(X),F^*)$ since $F^*$ is divisible.

The following theorem was suggested by James Lewis. For a prime $l\ne \car F$, write $e_l$ for the exponent of the torsion subgroup of the $l$-adic cohomology group $H^3(X,\Z_l)$. Then $e_l=1$ for almost all $l$: in characteristic $0$ this follows from comparison with Betti cohomology, and in characteristic $>0$ it is a famous theorem of Gabber \cite{gabber}. Set $e=\text{\rm lcm}(e_l)$: in characteristic $0$, $e$ is the exponent of $H^3_B(X,\Z)_\tors$, where $H^*_B$ denotes Betti cohomology.

\begin{Th} \label{t6.1} Assume that homological equivalence equals numerical equivalence on $CH_1(X)\otimes \Q$. Then, $e\bar \delta=0$  in Theorem \ref{t3.1}.
\end{Th}

\begin{rks}  1) This hypothesis holds if $\car F = 0$ by Lieberman \cite[Cor. 1]{lieberman}. His argument shows that, in characteristic $p$, it holds for $l$-adic cohomology if and only if the Tate conjecture holds for divisors on $X$ -- more correctly, for divisors on a model of $X$ over a finitely generated field. In particular, it holds if $X$ is an abelian variety; in this case, $e=1$.\\
2) The prime to the characteristic part of the unramified Brauer group also appears in the exact sequence of Theorem \ref{t3.1} as a Tate twist of the torsion of $H^1_\ind(X,\sK_2)$ \cite[Th. 1]{indec}.
\end{rks}

\begin{Th}\label{ts} Suppose $\dim X\le 2$ in Theorems \ref{t2.1} and \ref{t3.1}. Then there exists an integer $t>0$ such that
\begin{thlist}
\item $R^2_\nr\bG_m(X)\simeq D^1(X)$, $R^3_\nr\bG_m(X)\simeq\Ext_\Z(A_1^\alg(X),\Z)\simeq \break\Hom_\Z(\NS(X)_\tors,\Q/\Z)$ and  $tR^q_\nr\bG_m(X)=0$ for $q>3$.
\item  $tR^q_\nr\sK_2(X)=0$ for $q>3$ and $R^3_\nr\sK_2(X)=0$. Moreover,
if $\dim X=2$, the last map of Theorem \ref{t3.1} identifies 
$R^2_{\nr}\sK_2(X)$ with an extension of the Albanese kernel by a finite group.
\end{thlist}
We have $t=1$ if $\dim X<2$, and $t$ only depends on the Picard variety $\Pic^0_{X/F}$ if $\dim X=2$.
\end{Th}

Let $CH^2(X)_\alg$ denote the subgroup of $CH^2(X)$ consisting of cycle classes algebraically equivalent to $0$. Recall Murre's higher Abel-Jacobi map
\[AJ^3:CH^2(X)_\alg\to J^3(X)\]
where $J^3(X)$ is an algebraic intermediate Jacobian of $X$ \cite{murre}. Theorem \ref{ts} (ii)  suggests  that in general, $\IM(R^2_{\nr} (\sK_2)(X) \to CH^2 (X))$ should be contained in $\Ker AJ^3$. 

A key ingredient in the proofs of Theorems \ref{t2.1} and \ref{t3.1} is the work of \text{Ayoub} and Barbieri-Viale \cite{abv}, which identifies the ``maximal $0$-dim\-en\-sion\-al quotient" of the Nisnevich sheaf (with transfers) associated to the pre\-sheaf $U\mapsto CH^n(X\times U)$ with the group $A^n_\alg(X)$ of cycles modulo algebraic equivalence (see \eqref{eq2.1}).

The example $\sF=\sH^i_\et(\mu_m^{\otimes i})$ considered at the beginning relates to the sheaves studied in Theorems \ref{t2.1} and \ref{t3.1}  through the Bloch-Kato conjecture: Kummer theory for $\sK_1$ and the Merkurjev-Suslin theorem for $\sK_2$. Unfortunately, Theorem \ref{t2.1} barely suffices to compute $R^q_\nr(\bG_m/m)$ for $q\le 1$ and we have not been able to deduce from Theorem \ref{t3.1} any meaningful information on $R^*_\nr(\sK_2/m)$. We give the result for $R^1_\nr(\bG_m/m)$ without proof; there is an exact sequence, where $\NS(X)$ is the N\'eron-Severi group of $X$:
\[0\to (\NS(X)_{\text{tors}})/m\to R^1_\nr(\bG_m/m)\to {}_m D^1(X)\to 0\]
 and encourage the reader to test his or her insight on this issue.

Let us end this introduction by a comment on the content of the statement ``the assignment $X\mapsto \sF(X)$ makes $\sF$ an object of $\HI^\o$'', which applies to the objects appearing in Theorems \ref{t2.1} and \ref{t3.1}. It implies of course that $\sF(X)$ is a (stable) birational invariant of smooth projective varieties, which was already known in most cases; but it also implies some non-trivial functoriality, due to the additional structure of presheaf with transfers on $\sF$. For example, it yields a contravariant map $i^*:\sF(X)\to \sF(Y)$ for any closed immersion $i:Y\to X$. This does not seem easy to prove \emph{a priori}, say for $\sF(X)=D^1(X)=R^2_\nr\bG_m(X)_\tors$ in Theorem \ref{t2.1} (iii).

\section{Some results on birational motives}

We recall here some results from \cite{birat-tri}.

\begin{lemma}\label{l14.0} For any birational sheaf $\sF\in \HI^\o$ and any smooth variety $X$, $H^q(X,\sF)=0$ for $q>0$.
\end{lemma}

\begin{proof} See \cite[Prop. 1.3.3 b)]{birat-tri}.
\end{proof}

For the next proposition, let us write
\begin{equation}\label{eq1.1}
\nu^{\ge 1}M:=\uHom(\Z(1),M)(1)
\end{equation} 
for $M\in\DM^\eff$, where $\uHom$ is the internal Hom \cite[Prop. 3.2.8]{voetri}.

\begin{prop}\label{l5.3} For $M$ as above, we have a functorial exact triangle
\[\nu^{\ge 1}M\to M\to i^\o\nu_{\le 0} M\by{+1}.\]
Moreover, $M\in \IM i^\o$ if and only if $\uHom(\Z(1),M)=0$.
\end{prop}

\begin{proof} See \cite[Prop. 3.6.2 and Lemma 3.5.4]{birat-tri}.
\end{proof}

\begin{prop}\label{p8.3} For any $\sF\in \HI$, the counit map
\[i^\o R^0_\nr\sF \to \sF\]
is a monomorphism.
\end{prop}

\begin{proof} See \cite[Prop. 1.6.3]{birat-tri}.
\end{proof}

\begin{prop}\label{p2.3} Let $\sF\in \HI$. Then $\sF\in \HI^0$ if and only if $\sF_{-1}=0$, where $\sF_{-1}$ is the contraction of $\sF$ (\cite{voepre} or \cite[Lect. 23]{mvw}).
\end{prop}

\begin{proof} This is \cite[Prop. 1.5.2]{birat-tri}.
\end{proof}

\begin{prop}\label{p4.6} Let $C\in \DM^\eff$, and let $D=\uHom(\Z(1)[1],C)$. Then
\[\sH^i(D)=\sH^i(C)_{-1}\]
for any $i\in\Z$.
\end{prop}

\begin{proof} This is \cite[(4.1)]{birat-tri}.
\end{proof}

\section{Computational tools}\label{s2}

For $q\ge 0$, the $R^q_\nr$'s define a cohomological $\delta$-functor from $\HI$ to $\HI^\o$. Since $\HI$ is a Grothendieck
category (it has a set of generators and exact filtering direct limits),
it has enough injectives, so it makes sense to wonder if $R^q_\nr$ is the $q$-th derived functor of $R^0_\nr$. However, if $\sI\in \HI$
is injective, while $R^0_\nr\sI$ is clearly injective in $\HI^\o$, it is not clear whether $R^q_\nr\sI=0$ for $q>0$: the problem is similar to the one raised in \cite[Rk. 1 after Prop. 3.1.8]{voetri}. (In particular, the title of this paper should be taken with a pinch of salt.) Thus one cannot \emph{a priori} compute the higher $R^q_\nr$'s via injective resolutions; we give here another approach.

\begin{lemma}\label{l8.6} Let $\sF\in \HI$, and let $X$ be a smooth variety. Then the hypercohomology spectral sequence
\[E_2^{p,q}=H^p(X,R^q_\nr\sF)\Rightarrow H^{p+q}(X,R_\nr\sF)\]
degenerates, yielding isomorphisms
\[H^n(X,R_\nr\sF)\simeq H^0(X,R^n_\nr\sF).\]
\end{lemma}

\begin{proof} Indeed, $E_2^{p,q}=0$ for $p>0$ by Lemma \ref{l14.0}.
\end{proof}

\begin{prop}\label{p8.5} Let $C\in \DM^{\eff}$, and let $X$ be a smooth variety. Then we have a long exact sequence
\begin{multline*}
\dots\to H^n(X,R_\nr C)\to H^n(X,C)\\ 
\to \DM^\eff(\nu^{\ge 1}M(X),C[n])\to H^{n+1}(X,R_\nr C)\to\dots
\end{multline*}
In particular, if $C=\sF[0]$ for $\sF\in \HI$, we get a long exact sequence
\begin{multline*}
0\to R^0_\nr\sF(X)\to \sF(X)\to \DM^\eff(\nu^{\ge 1}M(X),\sF[0])\to\dots\\
\to R^n_\nr \sF(X)\to H^n(X,\sF)
\to \DM^\eff(\nu^{\ge 1}M(X),\sF[n])\to\dots
\end{multline*}
\end{prop}

\begin{proof} By iterated adjunction, we have
\begin{multline*}
H^n(X,R_\nr C)\simeq \DM^\eff(M(X),i^\o R_\nr C[n])\\
\simeq \DM^\o(\nu_{\le 0}M(X),R_\nr C[n])
\simeq \DM^\eff(i^\o\nu_{\le 0}M(X), C[n]).
\end{multline*}

The first exact sequence then follows from Proposition \ref{l5.3}. The second follows from the first, Lemma \ref{l8.6} and Proposition \ref{p8.3} b).
\end{proof}

\begin{prop}\label{p8.5a} Let $X$ be smooth and proper, and let $n\ge 0$. Then 
\[\uHom(\Z(n)[2n],M(X))\in (\DM^\eff)^{\le 0}.\]
Moreover,
\[\sH^0\left(\uHom(\Z(n)[2n],M(X)) \right)=\underline{CH}_n(X)\]
with
\[\underline{CH}_n(X)(U)=CH_n(X_{F(U)})\]
for any smooth connected variety $U$. Similarly, we have
\[\uHom(M(X),\Z(n)[2n])\in (\DM^\eff)^{\le 0}\]
and
\[\sH^0\left(\uHom(M(X),\Z(n)[2n]) \right)=\underline{CH}^n(X)\]
with
\[\underline{CH}^n(X)(U)=CH^n(X_{F(U)}).\]
\end{prop}

\begin{proof} The first statement is \cite[Th. 2.2]{motiftate}. The second is proven similarly.
\end{proof}

\begin{lemma}\label{l8.7} Let $\sF\in\HI^\o$. Then $R^q_\nr i^\o\sF=0$ for $q>0$.
\end{lemma}

\begin{proof} This is obvious from the adjunction isomorphism (due to the full faithfulness of $i^\o$) $\sF[0]\iso R_\nr i^\o\sF[0]$.
\end{proof}

\section{Varieties of dimension $\le 2$}

As in \cite[\S 3.4]{voetri}, let $d_{\le 0} \DM^\eff$ be the localising subcategory of $\DM^\eff$ generated by motives of varieties of dimension $0$: since $F$ is algebraically closed, this category is equivalent to the derived category $D(\Ab)$ of abelian groups \cite[Prop. 3.4.1]{voetri}. In \cite[Cor. 2.3.3]{abv}, Ayoub and Barbieri-Viale show that the inclusion functor
\[j:d_{\le 0} \DM^\eff\Inj \DM^\eff\]
has a left adjoint $L\pi_0$. 

\begin{lemma}\label{l3.2} a) For any smooth connected variety $X$, the structural map $X\to \Spec F$ induces an isomorphism $L\pi_0M(X)\iso L\pi_0\Z=\Z$.\\ 
b) We have $L\pi_0 \bG_m=0$.\\
c) If $C$ is a smooth projective irreducible curve with Jacobian $J$ (viewed as an object of $\HI$), then $L\pi_0J=0$.\\
d) If $A$ is an abelian variety, viewed as an object of $\HI$ (cf. \cite[Lemma 3.2]{spsz} or \cite[Lemma 1.4.4]{bar-kahn}), then there exists an integer $t>0$ such that $tL\pi_0 A=0$. Moreover, $L_0\pi_0(A):= H_0(L\pi_0(A))=0$.
\end{lemma}

\begin{proof} a) By adjunction and Yoneda's lemma, we have to show that for any object $C\in D(\Ab)$, the map
\[H^*_\Nis(F,C)\to H^*_\Nis(X,C)\]
is an isomorphism. This is well-known: by a hypercohomology spectral sequence, reduce to $C$ being a single abelian group; then $C$ is flasque (see \cite[Lemma 1.40]{riou}).

b) follows from a), applied to $X=\P^1$ (note that $M(\P^1)\simeq \Z\oplus \bG_m[1]$).

c) Let $M^0(C)$ be the fibre of the map $M(C)\to \Z$. By a), $L\pi M^0(C)=0$. By \cite[Th. 3.4.2]{voetri}, we have an exact triangle 
\begin{equation}\label{eq3.7}
 \bG_m[1]\to M^0(C)\to J[0]\by{+1}
 \end{equation}
so the claim follows from a) and b).

d) As is well-known, there exists a curve $C$ with Jacobian $J$ and an epimorphism $J\to A$, which is split up to some integer $t$ by complete reducibility. The first  claim then follows from c).

Let $\NST$ be the category of Nisnevich sheaves with transfers \cite{voetri}. To see that $L_0\pi_0(A)=0$, it is equivalent by adjunction to see that $\Hom_\NST(A,\sF)=0$ for any constant $\sF\in \NST$. We may identify $\sF$ with its value on any connected $X\in \Sm$. Let $f:A\to \sF$ be a morphism in $\NST$. Evaluating it on $1_A\in A(A)$, we get an element $f(1_A)\in \sF(A)=\sF$. If $X\in \Sm$ is connected and $a\in A(X)=\Hom_F(X,A)$, then $f(a)=a^*f(1_A)=f(1_A)$. So $f$ is constant, and since it is additive it must send $0$ to $0$. This proves that $f=0$, and thus $\Hom_\NST(A,\sF)=0$.
\end{proof}



\begin{prop}\label{p3.2}
Let $X/F$ be a smooth projective variety of dimension $\leq 2$. Then there exists an integer $t=t(X)>0$ such that $t\NS_1(X,i)=0$ for $i>0$. We have $t=1$ for $\dim X\le 1$, and we may take for $t$ the integer associated to $\Pic^0(X)$ in Lemma \ref{l3.2} d) for $\dim X=2$.
\end{prop}

\begin{proof}
Recall that $\NS_1(X,i):=H_i(L\pi_0\uHom(\Z(1)[2], M(X)))$ \cite[Def. 3.2.5]{abv}. For simplicity, write $C_X=\uHom(\Z(1)[2],M(X))$. We go case by case, using Poincar\'e duality as in \cite[Lemma B.1]{motiftate}:

If $\dim~X=0,$ then $X=\Spec F$ and hence $M(X)\simeq \Z$ is a birational motive; therefore $C_X=0$ (Proposition \ref{l5.3}) and $L\pi_0C_X=0$.

If $\dim X$=1, then Poincar\'e duality produces an isomorphism
\[
C_X\simeq \uHom(M(X),\Z)\simeq \Z[0].
\]

Hence $L\pi_0C_X=\Z[0]$.

Now suppose that $X$ is a smooth projective surface. By Poincar\'e duality, we get an isomorphism
\begin{equation}\label{eq3.8}
C_X \simeq \uHom(M(X),\Z(1)[2]).
\end{equation}

By evaluating the latter complex against a varying smooth variety, one computes its homology sheaves as $\Pic_{X/F}$ and $\bG_m$ in degrees $0$ and $1$
respectively and zero elsewhere.  Hence we have an exact triangle\footnote{It is split by the choice of a rational point of $X$, but this is useless for the proof.}
\begin{equation}\label{eq3.4}
 \bG_m[1]\to C_X \to \Pic_{X/F}[0]\by{+1}.
\end{equation}

We have $L\pi_0\bG_m[1]=0$  by Lemma \ref{l3.2} b). On the other hand, the representability of $\Pic_{X/F}$ yields an exact sequence 
\begin{equation}\label{eq3.5}
0\to \Pic^0_{X/F}\to \Pic_{X/F}\to \NS_{X/F}\to 0
\end{equation}
where $\Pic^0_{X/F}$ is the Picard variety of $X$ and $\NS_{X/F}$ is the (constant) sheaf of connected components of the group scheme $\Pic_{X/F}$. Hence an exact triangle
\[L\pi_0 \Pic^0_{X/F}\to L\pi_0\Pic_{X/F}\to L\pi_0\NS_{X/F}\by{+1}\]
where $L\pi_0\NS_{X/F}=\NS(X)$. By Lemma \ref{l3.2} d), $L\pi_0 \Pic^0_{X/F}$ is torsion, which concludes the proof. (The vanishing of $L_0\pi_0\Pic^0_{X/F}$ gives back the isomorphism $L_0\pi_0C_X\iso \NS(X)$ of \cite{abv}, see \eqref{eq2.1} below.)
\end{proof}

\section{Birational motives and indecomposable $(2,1)$-cycles}

In this section, we only assume $F$ perfect; we give proofs of two results promised in \cite[Rks 3.6.4 and 3.4.2]{birat-tri}.  These results are not used in the rest of the paper.

For the first one, let $X$ be a smooth projective variety, and let $M=\uHom(M(X),\Z(2)[4])$. Note that $M\simeq M(X)$ if $\dim X=2$ by Poincar\'e duality (cf. proof of Proposition \ref{p3.2}). The functor $\nu_{\le 0}$ is right $t$-exact as the left adjoint of the $t$-exact functor $i^\o$ \cite[Th. 3.4.1]{birat-tri}, so $\nu_{\le 0} M\in (\DM^\o)^{\le 0}$ since $M\in (\DM^\eff)^{\le 0}$ by Proposition \ref{p8.5a}. We want to compute the last two non-zero cohomology sheaves of $\nu_{\le 0} M$. Here is the result:

\begin{thm}\label{t4.1} With the above notation, we have 
\[\sH^i(\nu_{\le 0} M) =
\begin{cases}
\underline{CH}^2(X)& \text{for $i=0$}\\
\underline{H}^1_\ind(X,\sK_2)& \text{for $i=-1$}
\end{cases}
\]
where the sections of $\underline{H}^1_\ind(X,\sK_2)$ over a smooth connected $F$-variety $U$ with function field $K$ are given by the formula
\[\underline{H}^1_\ind(X,\sK_2)(U) = \Coker\left(\bigoplus_{[L:K]<\infty} \Pic(X_L)\otimes L^* \to H^1(X_K,\sK_2)\right)\]
in which the map is given by products and transfers.
\end{thm}

\begin{proof} We use the exact triangle of Proposition \ref{l5.3}. From the cancellation theorem (\cite{voecan}, \cite[Prop. A.1]{motiftate}), we get an isomorphism 
\[\nu^{\ge 1} M\simeq \uHom(M(X),\Z(1)[4])(1)\simeq C_X\otimes \bG_m[1]\]
where $C_X=\uHom(M(X),\Z(1)[2])$.

By Proposition \ref{p8.5a}, $C_X\in (\DM^\eff)^{\le 0}$. On the other hand, $\otimes$ is right $t$-exact because it is induced by a right $t$-exact $\otimes$-functor on $D(\NST)$ via the right $t$-exact functor $LC:D(\NST)\to \DM^\eff$. Hence $\nu^{\ge 1}M\in (\DM^\eff)^{\le -1}$. 

Using Proposition \ref{p8.5a} again, this shows the assertion in the case $i= 0$ (compare \cite[Th. 2.2 and its proof]{motiftate}). For the case $i=-1$, the long exact sequence of cohomology sheaves yields an exact sequence:
\[\dots \to \sH^0(C_X\otimes \bG_m)\to \sH^{-1}(M)\to \sH^{-1}(i^\o\nu_{\le 0} M)\to 0.\]

Let $\sF=\sH^0(C_X) = \underline{CH}^1(X)$; then $\sH^0(C_X\otimes \bG_m)=\sF\otimes_{\HI} \bG_m$ by right $t$-exactness of $\otimes$; here $\otimes_{\HI}$ is the tensor structure induced by $\otimes$ on $\HI$. For any function field $K/F$, the map induced by transfers 
\[\bigoplus_{[L:K]<\infty} \sF(L)\otimes \bG_m(L)\to (\sF\otimes_{\HI} \bG_m)(K)\]
is surjective \cite[2.14]{somekawa}, which concludes the proof.
\end{proof}

The second result which was promised in \cite[Rk. 4.3.2]{birat-tri} is:

\begin{prop} Let $E$ be an elliptic curve over $F$. Then the sheaf
\[\Tor_1^{\DM}(E,E) := \sH^{-1}(E[0]\otimes E[0])\]
is not birational. Here $E$ is viewed as an object of $\HI$ \cite[Lemma 1.4.4]{bar-kahn}.
\end{prop}

(This contrasts with the fact that the tensor product of two birational sheaves is birational, \cite[Th. 4.3.1]{birat-tri}.)

\begin{proof} Up to extending scalars, we may and do assume that $\End(E)=\End(E_{\bar F})$. Consider the surface $X=E\times E$. The choice of the rational point $0\in E$ yields a Chow-K\"unneth decomposition of the Chow motive of $E$, hence by \cite[Prop. 2.1.4]{voetri} an isomorphism
\[M(E)\simeq \Z[0]\oplus E[0]\oplus\Z(1)[2]  \]
(compare also \cite[Th. 3.4.2]{voetri}). Therefore
\[M(X)\simeq \Z[0]\oplus 2E[0]\oplus 2\Z(1)[2] \oplus E[0]\otimes E[0]\oplus 2E(1)[2] \oplus  \Z(2)[4].\]

This allows us to compute $\uHom(\Z(1)[2],E[0]\otimes E[0])$ as a direct summand of $\uHom(\Z(1)[2], M(X))=C_X$. First we have 
\[
\uHom(\Z(1)[2],\Z[0]) =\uHom(\Z(1)[2],E[0])=0.
\]

The first vanishing is \cite[Lemma A.2]{motiftate}, while the second one follows from the Poincar\'e duality isomorphism $\uHom(\Z(1)[2],M(E))\simeq \uHom(M(E),\Z)=\Z$ \cite[Lemma 2.1 a)]{gcell}. Hence, using the cancellation theorem:
\[C_X\simeq 2\Z[0]\oplus  \uHom(\Z(1)[2],E[0]\otimes E[0])\oplus 2E[0] \oplus\Z(1)[2]\]
and
\[\Pic_{X/F}=\sH^0(C_X)\simeq 2\Z\oplus \sH^0(\uHom(\Z(1)[2],E[0]\otimes E[0]))\oplus 2E.\]

On the other hand, using Weil's formula for the Picard group of a product, we have a canonical decomposition
\[\Pic_{E\times E/F} \simeq \Pic^0_{E\times E/F} \oplus \NS(E)\oplus \NS(E) \oplus \Hom(E,E)=2E\oplus 2\Z\oplus \End(E).\]

One checks that the idempotents involved in the two decompositions of $\Pic_{X/F}$ match to yield an isomorphism
\[\sH^0(\uHom(\Z(1)[2],E[0]\otimes E[0]))\simeq \End(E)\]
where $\End(E)$ is viewed as a constant sheaf. By the $t$-exactness of Voevodsky's contraction functor $(-)_{-1}=\uHom(\bG_m,-)$ \cite[Prop. 4.1.1]{birat-tri}, this yields an isomorphism $\End(E)\iso \Tor_1^{\DM}(E,E)_{-1}$, which proves that $\Tor_1^{\DM}(E,E)$ is not birational (see Proposition \ref{p2.3}).
\end{proof}

\section{The  case of $\bG_m$: proof of Theorems \ref{t2.1}, \ref{t2.1bis} and \ref{ts} {\rm (i)}}

\subsection{Proof of Theorem \ref{t2.1}}
We apply Proposition \ref{p8.5} to $\sF=\bG_m$. The Nisnevich cohomology of $\bG_m$ is well-known: we have
\[H^n(X,\bG_m)=
\begin{cases}
F^* &\text{if $n=0$}\\
\Pic(X)&\text{if $n=1$}\\ 
0 &\text{otherwise.}
\end{cases}
\]

Noting that $\bG_m[0]=\Z(1)[1]$ in $\DM^\eff$, we get
\begin{multline*}
\DM^\eff(\nu^{\ge 1}M(X),\bG_m[n])=\\
\DM^\eff(\uHom(\Z(1),M(X))(1),\Z(1)[n+1])=\\
\DM^\eff(\uHom(\Z(1)[2],M(X)),\Z[n-1])
\end{multline*}
by using the cancellation theorem.
Thus
\begin{multline} \label{eq8.3}
\DM^\eff(\uHom(\Z(1)[2],M(X)),\Z[n-1])\\
\simeq d_{\le 0}\DM^\eff(L\pi_0\uHom(\Z(1)[2],M(X)),\Z[n-1])\\
=D(\Ab)(L\pi_0\uHom(\Z(1)[2],M(X)),\Z[n-1])=:F_n(X).
\end{multline}

The homology group $H_s(L\pi_0\uHom(\Z(1)[2],M(X)))$ is denoted by $\NS_1(X,s)$ in \cite[3.25]{abv}. The universal coefficients  theorem then gives  an exact sequence
\begin{multline}\label{eq4.4}
0 \to \Ext_{\Ab}(\NS_1(X,n-2),\Z) \to F_n(X)\\ \to \Ab(\NS_1(X,n-1),\Z) \to 0.
\end{multline}

By Proposition \ref{p8.5a}, $\uHom(\Z(1)[2],M(X))\in (\DM^\eff)^{\le 0}$. Since the inclusion functor $j$ is $t$-exact, $L\pi_0$ is right $t$-exact by a general result on triangulated categories \cite[Prop. 1.3.17]{bbd}, hence $\NS_1(X,n)=0$ for $n< 0$. For $n=0$, Ayoub and Barbieri-Viale find
\begin{equation}\label{eq2.1}
\NS_1(X,0)=A_1^\alg(X)
\end{equation}
in \cite[Th. 3.1.4]{abv}\footnote{The hypothesis $F$ algebraically closed is sufficient for their proof.}.

Gathering all this, we get (i) (which also follows from \eqref{eq1}), an exact sequence
\begin{equation}\label{eq8.4}
0\to R^1_\nr\bG_m(X)\to \Pic(X)\by{\delta} \Hom(A_1^\alg(X),\Z)\to R^2_\nr\bG_m(X)\to 0
\end{equation}
and isomorphisms
\begin{equation}\label{eq4.2}
F_n(X)\iso  R_\nr^{n+1}\bG_m(X)
\end{equation}
for $n\ge 2$, which yield (iv) thanks to \eqref{eq8.4}.

In  Lemma  \ref{l8.8} below, we shall show that $\delta$ is induced by the intersection pairing. Granting this for the moment, (ii) is immediate and we get a cross of exact sequences
\[\xymatrix{
&&0\ar[d]\\
&&\Hom(N_1(X),\Z)\ar[d]\\
0\ar[r]& N^1(X)\ar[r]\ar[ru]& \Hom(A_1^\alg(X),\Z)\ar[r]\ar[d]& R^2_\nr\bG_m(X)\ar[r]& 0\\
&&\Hom(\Griff_1(X),\Z)\ar[d]\\
&&0
}\]
in which the triangle commutes, and where we used that $N_1(X)$ is a free finitely generated abelian group. The exact sequence of (iii) then follows from a diagram chase.

\begin{lemma}\label{l8.8}
The map $\delta$ of \eqref{eq8.4} is induced by the intersection pairing. 
\end{lemma}

\begin{proof}
This map comes from the composition
\begin{multline}\label{eq8.5}
\DM^\eff(M(X),\Z(1)[2])\\
\to \DM^\eff(\uHom(\Z(1)[2],M(X))(1)[2],\Z(1)[2])\\
= \DM^\eff(\uHom(\Z(1)[2],M(X)),\Z)\\
\to \Hom_\Z(\Hom(\Z(1)[2],M(X)),\Z)
\end{multline}
in which the first map is induced by the canonical morphism $\nu^{\ge 1}M(X)\allowbreak\to M(X)$, the equality follows from the cancellation theorem \cite{voecan} and the third is by taking global sections at $\Spec k$. 

Consider the natural pairing
\begin{multline*}
\uHom(M(X),\Z(1)[2])\otimes \uHom(\Z(1)[2],M(X))\\
\to \uHom(\Z(1)[2],\Z(1)[2])=\Z[0].
\end{multline*}

By Proposition \ref{p8.5a}, this pairing factors through a pairing
\[\underline{CH}^1(X)[0]\otimes \underline{CH}_1(X)[0]\to \Z[0].\]

Taking global sections, we clearly get the intersection pairing.

From the above,  we get a commutative diagram
\[\begin{CD}
\uHom(M(X),\Z(1)[2])@>>> \uHom(\uHom(\Z(1)[2],M(X)),\Z[0])\\
@V{}VV @A{}AA\\
\underline{CH}^1(X)[0]@>>> \uHom(\underline{CH}_1(X)[0],\Z[0]).
\end{CD}\]

Applying the functor $\DM^\eff(\Z,-)$ to this diagam, we get a commutative diagram of abelian groups
\[\begin{CD}
\DM^\eff(M(X),\Z(1)[2])@>a>> \DM^\eff(\uHom(\Z(1)[2],M(X)),\Z[0])\\
@V{}VV @A{b}AA\\
CH^1(X)@>>> \DM^\eff(\underline{CH}_1(X)[0],\Z[0]).
\end{CD}\]

In this diagram,  one checks easily that $a$ corresponds to \eqref{eq8.5} via the cancellation theorem. On the other hand,  $b$ is an isomorphism. Now  the evaluation functor at $\Spec F$, $\sF\mapsto \sF(F)$, yields a commutative triangle
\[\xymatrix{
CH^1(X)\ar[r]\ar[dr]_{\cap}& \DM^\eff(\underline{CH}_1(X)[0],\Z[0])\ar[d]_{ev_F}\\
&\Hom(CH_1(X),\Z).
}\]
where $\cap$ is the intersection pairing (see above). But we saw that $\DM^\eff(\underline{CH}_1(X)[0],\Z[0])\simeq \Hom(A_1^\alg(X),\Z)$ (\eqref{eq8.3}, \eqref{eq4.4} and  \eqref{eq2.1}); via this isomorphism, $ev_F$ is induced by the surjection $CH^1(X)\surj A^1_\alg(X)$, hence is injective. This concludes the proof.
\end{proof}

\subsection{Proof of Theorem \ref{t2.1bis}} We use the following lemma:

\begin{lemma} \label{l4.1} In $\DM^\eff$, the map $\bG_m\to \bG_m^\et$ is an isomorphism on $H^0$; moreover, 
$R^1\alpha_*\alpha^*\bG_m=0$ and $R^2\alpha_*\alpha^*\bG_m$ is the Nisnevich sheaf $\Br$ associated to the presheaf $U\mapsto \Br(U)$\footnote{This presheaf is in fact already a Nisnevich sheaf.}. Here, $\alpha:\Sm_\et\to \Sm_\Nis$ is the change of topology morphism.
\end{lemma}

\begin{proof} The first statement is obvious, the second one follows from the local vanishing of $\Pic$ and the third one is tautological. 
\end{proof}

To compute $R_\nr \bG_m^\et$, we may use the ``hypercohomology'' spectral sequence
\[E_2^{p,q} = R^p_\nr R^q\alpha_*\alpha^*\bG_m\Rightarrow R^{p+q}_\nr\bG_m^\et.\]

From Lemma \ref{l4.1}, we find an isomorphism
\[R^1_\nr \bG_m\iso R^1_\nr\bG_m^\et\]
and a five term exact sequence
\[0\to R^2_\nr \bG_m\to R^2_\nr \bG_m^\et \to R^0_\nr \Br \to R^3_\nr \bG_m\to R^3_\nr \bG_m^\et\]
which yields (a more precise form of) Theorem \ref{t2.1bis} in view of the obvious isomorphism $R^0_\nr \Br=\Br_\nr$, where $\Br_\nr$ is the unramified Brauer group.

\subsection{Proof of Theorem \ref{ts} (i)} Since $\dim X\le 2$, $\Griff_1(X)$ is torsion hence $\Hom(\Griff_1(X),\Z)=0$, which gives the first statement. Then, Theorem \ref{t2.1} (iv) and Proposition \ref{p3.2} yield isomorphisms
\[\Ext_\Z(\NS_1(X,q-3),\Z)\iso R^q_\nr\bG_m(X),\quad q\ge 3.\]

For $q>3$, the left hand group is killed by the integer $t$ of Proposition \ref{p3.2}. Suppose $q=3$; then $\NS_1(X,q-3)=A_1^\alg(X)$, which proves Theorem \ref{ts} (i) except for the isomorphism involving $\NS(X)°_\tors$. For this we distinguish 3 cases: 
\begin{enumerate}
\item If $\dim X =0$, $A_1^\alg(X)=\NS(X)=0$ and the statement is true.
\item If $\dim X=1$, $A_1^\alg(X)\simeq \Z\simeq \NS(X)$ and the statement is still true.
\item If $\dim X=2$, $A_1^\alg(X)=\NS(X)$. But for any finitely generated abelian group $A$, there is a string of canonical isomorphisms
\[\Ext_\Z(A,\Z)\iso \Ext_\Z(A_\tors,\Z)\osi \Hom_\Z(A_\tors,\Q/\Z). \]
\end{enumerate}

This concludes the proof.

\section{The case of $\sK_2$: proof of Theorems \ref{t3.1} and \ref{ts} {\rm (ii)}}\label{sK2}

\subsection{Preparations}

\begin{lemma}\label{l3.1} a) The natural map
\begin{equation}\label{eq3.1}\Z(2)[2]\to \sK_2[0]
\end{equation}
induces an isomorphism
\[\cone\left(i^\o R_\nr \Z(2)[2]\to \Z(2)[2]\right)\iso \cone\left(i^\o R_\nr \sK_2[0]\to \sK_2[0]\right). \]
b) The map \eqref{eq3.1} induces an isomorphism
\[\DM^\eff(\nu^{\ge 1} C,\Z(2)[2])\iso \DM^\eff(\nu^{\ge 1} C,\sK_2[0])\]
for any $C\in \DM^\eff$. (See \eqref{eq1.1} for the definition of $\nu^{\ge 1} C$.)
\end{lemma}

\begin{proof} By the cancellation theorem, we have
\[ \uHom(\Z(1)[1],\Z(2)[2])\simeq \Z(1)[1]\simeq \bG_m[0]\]
in $\DM^\eff$. 

Let $\sH^i(C)$ denote the $i$-th cohomology sheaf of an object $C\in \DM^\eff$.  By Proposition \ref{p4.6}, the $i$-th cohomology sheaf of the left hand side is $\sH^i(\Z(2)[2])_{-1}$. Thus the latter sheaf is $0$ for $i\ne 0$. By Proposition \ref{p2.3}, $\sH^i(\Z(2)[2])\in \HI^\o$ for $i\ne 0$, hence $\tau_{<0}(\Z(2)[2])\in \DM^\o$. By adjunction, we deduce
\[\cone\left(i^\o R_\nr\tau_{<0}(\Z(2)[2])\to \tau_{<0}(\Z(2)[2])\right)=0\]
which in turn implies a).

To pass from a) to b), use the fact that, for $C,D\in \DM^\eff$, adjunction transforms the exact sequence
\[\DM^\eff(i^\o\nu_{\le 0} C,D)\to \DM^\eff(C,D)\to \DM^\eff(\nu^{\ge 1} C,D)\]
into the exact sequence 
\begin{multline*}
\DM^\eff(C,i^\o R_\nr D)\to \DM^\eff(C,D)\\
\to \DM^\eff(C,\cone(i^\o R_\nr D\to D)).
\end{multline*}
\end{proof}

Applying the exact sequence of Proposition \ref{p8.5} to $C=\sK_2[0]$ and using Lemma \ref{l3.1} b), we get a long exact sequence
\begin{multline*}
\dots\to H^{n}(X,R_\nr \sK_2)\to H^{n}(X,\sK_2)\\ 
\to \DM^\eff(\nu^{\ge 1}M(X),\Z(2)[n+2])\to H^{n+1}(X,R_\nr \sK_2)\to\dots
\end{multline*}

Using the cancellation theorem, we get an isomorphism
\[\DM^\eff(\nu^{\ge 1}M(X),\Z(2)[n+2])\simeq \DM^\eff(\uHom(\Z(1)[2],M(X)),\Z(1)[n]).\]

Since  $\Z(1)[n]=\bG_m[n-1]$, using Lemma \ref{l8.6} we get an exact sequence
\begin{multline}\label{eq3.2}
0\to (R^1_\nr \sK_2)(X)\to H^{1}(X,\sK_2)\\ 
\by{\delta} \DM^\eff(\uHom(\Z(1)[2],M(X)),\bG_m[0])\to (R^2_\nr \sK_2)(X)\to CH^{2}(X)\\
\by{\phi} \DM^\eff(\uHom(\Z(1)[2],M(X)),\bG_m[1])\to (R^3_\nr \sK_2)(X)\to 0
\end{multline}
and isomorphisms for  $q>3$
\begin{equation}\label{eq3.6}
 \DM^\eff(\uHom(\Z(1)[2],M(X)),\bG_m[q-2])\iso (R^q_\nr \sK_2)(X)
\end{equation}
 where we also used that $H^2(X,\sK_2)\simeq CH^2(X)$ and $H^i(X,\sK_2)=0$ for $i>2$.

\subsection{Proof of Theorem \ref{t3.1}}

The group $\DM^\eff(\uHom(\Z(1)[2],M(X)),\bG_m[0])$ may be computed as follows:
\begin{multline}\label{eq3.3}
\DM^\eff(\uHom(\Z(1)[2],M(X)),\bG_m[0])\\
\overset{1}{\simeq} \HI(\sH_0(\uHom(\Z(1)[2],M(X))),\bG_m)\\
\overset{2}{\simeq} \HI(\underline{CH}_1(X),\bG_m)\overset{3}{\simeq} \HI^\o(\underline{CH}_1(X),R^0_\nr \bG_m)\\
\overset{4}{\simeq} \HI(\underline{CH}_1(X),j F^*)\overset{5}{\simeq} \Ab(L_0\pi_0\underline{CH}_1(X), F^*)\\
\overset{6}{\simeq} \Ab(A_1^\alg(X), F^*).
\end{multline}

Here, isomorphism 1 follows from the fact that $\uHom(\Z(1)[2],M(X))\allowbreak\in (\DM^\eff)^{\le 0}$ (Proposition \ref{p8.5a}), 2 comes from the computation of $\sH_0$ (ibid.), 3 follows from adjunction, knowing that $\underline{CH}_1(X)$ is a birational sheaf (ibid.), 4 follows from Theorem \ref{t2.1} (i),   5 comes from adjunction and 6 follows from \eqref{eq2.1}.

Thus the homomorphism $\delta$ corresponds to a pairing
\[H^{1}(X,\sK_2)\times A_1^\alg(X)\to F^*.\]

Let $d=\dim X$. An argument analogous to that in the proof of Lemma \ref{l8.8} shows that this pairing comes from the ``intersection'' pairing
\begin{multline}\label{eq4.1}
H^3(X,\Z(2))\times H^{2d-2}(X,\Z(d-1))\by{\cap} H^{2d+1}(X,\Z(d+1))\\
\by{\pi_*} H^1(F,\Z(1))=F^*
\end{multline}
where the last map is induced by the ``Gysin'' morphism ${}^t\pi:\Z(d)[2d]\allowbreak\to M(X)$. Here we used the isomorphisms
\[H^{1}(X,\sK_2)\simeq H^3(X,\Z(2)),\quad CH_1(X)\simeq H^{2d-2}(X,\Z(d-1)).\]

In particular, \eqref{eq4.1} factors through algebraic equivalence. This was proven by Coombes \cite[Cor. 2.14]{coombes} in the special case of a surface; we shall give a different proof below, which avoids the use of \eqref{eq2.1}.

Consider the product map
\[c:CH^1(X)\otimes F^*=H^1(X,\sK_1)\otimes H^0(X,\sK_1)\to H^1(X,\sK_2).\]
By functoriality, we have a commutative diagram of pairings
\[\begin{CD}
CH^1(X)\otimes F^*\times A_1^\alg(X)@>>> F^*\\
@V{c\times 1}VV ||\\
H^{1}(X,\sK_2)\times A_1^\alg(X)@>>> F^*
\end{CD}\]
where the top pairing is the intersection pairing $CH^1(X)\times A_1^\alg(X)\to \Z$, tensored with $F^*$. Since the latter is $0$ when restricted to $\Griff_1(X)$, we get an induced pairing
\[H_\ind^{1}(X,\sK_2)\times \Griff_1(X)\to F^*\]
yielding a commutative diagram
\[\begin{CD}
&&&&&&0\\
&&&&&&@VVV\\
0&\to&\Pic^\tau(X)\otimes F^* @>>> \Pic(X)\otimes F^* @>\alpha>> \Hom(A_1^\num(X),F^*)\\
&&&& @VVV @VVV\\
0&\to&(R^1_\nr \sK_2)(X)@>>> H^{1}(X,\sK_2) @>{\delta}>> \Hom(A_1^\alg(X), F^*)\\
&&&& @VVV @VVV\\
&&&& H_\ind^{1}(X,\sK_2) @>{\bar\delta}>> \Hom(\Griff_1(X), F^*).\\
&&&&@VVV@VVV\\
&&&&0&&0
\end{CD}\]


In this diagram, all rows and columns are complexes. The middle row and the two columns are exact; moreover, $\alpha$ is surjective as one sees by tensoring with $F^*$ the exact sequence
\[0\to \Pic^\tau(X)\to \Pic(X)\to \Hom(A_1^\num(X),\Z)\to D^1(X)\to 0.\]

Then a diagram chase yields an exact sequence
\[\Pic^\tau(X)\otimes F^*\to (R^1_\nr \sK_2)(X)\to H_\ind^{1}(X,\sK_2) \by{\bar\delta} \Hom(\Griff_1(X), F^*)\]
and the surjectivity of $\alpha$ implies that the map $\Hom(A_1^\alg(X), F^*)\to (R^2_\nr \sK_2)(X)$ given by \eqref{eq3.2} and \eqref{eq3.3} factors through $\Hom(\Griff_1(X), F^*)$. This concludes the proof.

\subsection{Direct proof that \eqref{eq4.1} factors through algebraic equivalence} Consider \allowbreak clas\-ses $\alpha\in H^3(X,\Z(2))$ and $\beta\in CH^{d-1}(X)$: assuming that $\beta$ is algebraically equivalent to $0$, we must prove that $\pi_*(\alpha\cdot \beta)=0$, where $\pi$ is the projection $X\to \Spec F$.

By hypothesis, there exists a smooth projective curve $C$, two points $c_0,c_1\in C$ and a cycle class $\gamma\in CH^{d-1}(X\times C)$ such that $\beta=c_0^*\gamma - c_1^*\gamma$. Let $\pi_X:X\times C \to X$ and $\pi_C:X\times C\to C$ be the two projections. 

The Gysin morphism ${}^t\pi:\Z(d)[2d]\to M(X)$ used in the definition of \eqref{eq4.1} extends trivially to give morphisms $M(d)[2d]\to M\otimes M(X)$ for any $M\in \DM^\eff$, which are clearly natural in $M$: this applies in particular to $M=M(C)$, giving a Gysin morphism ${}^t\pi_C:M(C)(d)[2d]\to M(X\times C)$ which induces a map 
\[(\pi_C)_*:H^{2d+1}(X\times C,\Z(d+1))\to H^1(C,\Z(1)).\]

The naturality of these Gysin morphisms then gives
\begin{multline*}\pi_*(\alpha\cdot \beta)= \pi_*(\alpha\cdot (c_0^*\gamma-c_1^*\gamma)))\\
= \pi_*(c_0^*(\pi_X^*\alpha\cdot \gamma)-c_1^*(\pi_X^*\alpha\cdot \gamma))= (c_0^*-c_1^*)(\pi_C)_*(\pi_X^*\alpha\cdot \gamma).
\end{multline*}

But $c_i^*:H^1(C,\Z(1))\to H^1(F,\Z(1))$ is left inverse to ${\pi'}^*:H^1(F,\Z(1))\allowbreak\to H^1(C,\Z(1))$ (where $\pi':C\to \Spec F$ is the structural map), which is an isomorphism since $C$ is proper. Hence $c_0^*=c_1^*$ on $H^1(C,\Z(1))$,  and the proof is complete.

\subsection{Proof of Theorem \ref{ts} {\rm (ii)}} Note that $\Griff_1(X)$ is finite if $\dim X\le 2$. In view of \eqref{eq3.2} and \eqref{eq3.6}, it therefore suffices to prove

\begin{prop} a) If $\dim X\le 2$, we have 
\[t \DM^\eff(\uHom(\Z(1)[2],M(X)),\bG_m[i])=0\] 
for $i>1$, and also for $i=1$ if $\dim X<2$.\\
b) Suppose $\dim X=2$. Then the map $\phi$ of \eqref{eq3.2} is the Albanese map from \cite[(8.1.1)]{birat-pure}.
\end{prop}

a) is a d\'evissage  similar to the one for Proposition \ref{p3.2} (using \eqref{eq3.4} and \eqref{eq3.5} for $\dim X=2$); we leave details to the reader. As for b), we have a diagram in $\DM^\eff$
\begin{equation}\label{eq6.4}
\begin{CD}
\uHom(M(X),\Z(2)[4])&\longrightarrow& \uHom(\uHom(\Z(1)[2],M(X)),\Z(1)[2])\\
@A \Delta A\wr A @A\Delta^* A\wr A\\
M(X)&\by{\epsilon_X}& \uHom(\uHom(M(X),\Z(1)[2]),\Z(1)[2])
\end{CD}
\end{equation}
defined as follows. The top row is obtained by applying $\uHom(-,\Z(2)[4])$ to the map $\nu^{\ge 1} M(X)\to M(X)$ of Proposition \ref{l5.3}, and using the cancellation theorem. The bottom row is obtained by adjunction from the evaluation morphism $M(X)\otimes \uHom(M(X),\Z(1))\to \Z(1)$. The Poincar\'e duality isomorphism $\Delta$ is induced by adjunction by the map
\[M(X\times X)\simeq M(X)\otimes M(X)\to \Z(2)[4]\]
defined by the class of the diagonal $\Delta_X\in CH^2(X\times X) = \DM^\eff(M(X\times X),\Z(2)[4])$ (see \cite[Prop. 2.5.4]{bar-kahn}). The isomorphism $\Delta^*$ is induced by the isomorphism $\uHom(\Z(1)[2],M(X))\iso \uHom(M(X),\Z(1)[2])$ of \eqref{eq3.8}, deduced by adjunction from the composition
\begin{multline*}
\uHom(\Z(1)[2],M(X))\otimes M(X) \to \uHom(\Z(1)[2],M(X)\otimes M(X))\\ \by{(\Delta_X)_*}\uHom(\Z(1)[2],\Z(2)[4])\simeq \Z(1)[2]
\end{multline*}
where the last isomorphism follows again from the cancellation theorem.\footnote{Note that evaluation and adjunction yield a tautological morphism $\uHom(A,B)\otimes C)\to \uHom(A,B\otimes C)$ for $A,B,C\in \DM^\eff$.} A tedious but trivial bookkeeping yields:

\begin{lemma}\label{l6.3} The diagram \eqref{eq6.4} commutes.\qed
\end{lemma}

We are therefore left to identify $\DM^\eff(\Z,\epsilon_X)$ (where $\epsilon_X$ is as in \eqref{eq6.4}) with the Albanese map. 
For simplicity, let us write in the sequel $\sF$ rather than $\sF[0]$ for a sheaf $\sF\in \HI$ placed in degree $0$ in $\DM^\eff$.
Let $\sA_X$ be the Albanese scheme of $X$ in the sense of Serre-Ramachandran, and let $a_X:M(X)\to \sA_X$ be the map defined by \cite[(7)]{spsz}. On the other hand, write $D$ for the (contravariant) endofunctor $M\mapsto \uHom(M,\bG_m[1])$ of $\DM^\eff$, and $\epsilon:Id_{\DM^\eff}\Rightarrow D^2$ for the biduality morphism, so that $\epsilon_X=\epsilon_{M(X)}$. We get a commutative diagram:
\begin{equation}\label{eq6.5}
\begin{CD}
M(X)@>\epsilon_{M(X)}>> D^2 M(X)\\
@V a_X VV @VD^2(a_X)VV\\
\sA_X@>\epsilon_{\sA_X}>> D^2 \sA_X
\end{CD}
\end{equation}

 It is sufficient to show:

\begin{prop}\label{p6.1} After application of $\DM^\eff(\Z,-)=H^0_\Nis(k,-)$ to \eqref{eq6.5}, we get a commutative diagram
\[\begin{CD}
CH_0(X) @>>> \sA_X(k)\\
@V a_X(k) VV @Vu VV\\
\sA_X(k) @>v >> \sA_X(k)\oplus Q
\end{CD}\]
where $a_X(k)$ is the Albanese map, $Q$ is some abelian group and $u,v$ are the canonical injections.
\end{prop}

The main lemma is:

\begin{lemma}\label{l6.4} Let $A$ be an abelian $F$-variety. Then there is a canonical isomorphism
\[D A\simeq A^*\oplus \tau_{\ge 2} DA\]
where $A^*$ is the dual abelian variety of $A$.\\ 
\end{lemma}

\begin{proof} Note that \eqref{eq3.4} holds for any smooth projective variety $Y$, if we replace $C_Y$ by $D (M(Y))$. We shall take $Y=A$ and $Y=A\times A$. Let $p_1,p_2,m:A\times A\to A$ be respectively the first and second projection and the multiplication map. The composition
\[M(A\times A)\by{(p_1)_*+(p_2)_*-m_*} M(A)\by{a_A} \sA_A\]
is $0$. One characterisation of $\Pic^0_{A/F}\subset \Pic_{A/F}$ is as the kernel of $(p_1)^*+(p_2)^*-m^*$ (e.g. \cite[\S\ before Rk. 9.3]{milne}). Therefore, the composition
\[D A\to D \sA_A \by{D(a_A)}D (M(A)) \by{\eqref{eq3.4}} \Pic_{A/F}\]
induces a morphism
\begin{equation}\label{eq6.6}
D A\to \Pic^0_{A/F} = A^*.
\end{equation}

Here we used the canonical splitting of the extension
\[
0\to A\to \sA_A\to \Z\to 0
\]
given by the choice of the origin $0\in A$. In view of the exact triangle
\[\tau_{\le 1} DA \to DA\to \tau_{\ge 2} DA\by{+1},\]
to prove the lemma we have to show that \eqref{eq6.6} becomes an isomorphism after applying the truncation $\tau_{\le 1}$ to its left hand side.

For this, we may evaluate on smooth $k$-varieties, or even on their function fields $K$ by ``Gersten's principle'' \cite[\S 2.4]{bar-kahn}. For such $K$, we have to show that the homomorphism
\[\Ext^{1+i}_\NST(A_K,\bG_m)\to H^i_\Nis(K,A^*) \]
is an isomorphism for $i\le 1$. This is clear for $i<-1$. For $i=-1,0,1$, let $\EST$ be the category of \'etale sheaves with transfers of \cite[Lect. 6]{mvw}, and $\ES$ the category of sheaves of abelian groups on $\Sm_\et$, so that we have exact functors
\[\NST\by{\alpha^*} \EST \by{\omega} \ES\]
where $\alpha^*$ is \'etale sheafification and $\omega$ is  ``forgetting transfers''. If $\alpha_*$ denotes the right adjoint of $\alpha^*$, the hyperext spectral sequence
\[E_2^{p,q}=\Ext^p_\NST(A_K,R^q\alpha_*\alpha^* \bG_m)\Rightarrow \Ext^{p+q}_\EST(\alpha^*A_K,\alpha^*\bG_m)\]
and the vanishing of $R^1\alpha_*\alpha^*\bG_m$ (Hilbert 90!) yield isomorphisms
\[\Ext^{1+i}_\NST(A_K,\bG_m)\iso \Ext^{1+i}_\EST(\alpha^*A_K,\alpha^*\bG_m), \quad i\le 0 \]
and an injection
\[\Ext^{2}_\NST(A_K,\bG_m)\inj \Ext^{2}_\EST(\alpha^*A_K,\alpha^*\bG_m). \]

Finally, by \cite[Th. 3.14.2 a)]{bar-kahn}, we have an isomorphism
\[\Ext^i_\EST(\sF_K,\sG)\iso \Ext^i_\ES(\omega \sF_K, \omega\sG)\]
when $\sF,\sG\in \EST$ are ``$1$-motivic'', e.g. $\sF=\alpha^*A,\sG=\alpha^*\bG_m$; moreover, these groups vanish for $i\ge 2$. Lemma \ref{l6.4} now follows from the obvious vanishing of $H^1_\Nis(K,A^*)$, the vanishing of $\Hom_\ES(A_K,\bG_m)$ and the isomorphism
\[\Ext^1_\ES(A_K,\bG_m)\iso A^*(K)\]
deduced from  the Weil-Barsotti formula. 
%
%
%
\end{proof}

\enlargethispage*{30pt}

\begin{proof}[Proof of Proposition \ref{p6.1}] Let $A=\sA^0_{X/F}$ be the Albanese variety of $X$. Lemma  \ref{l6.4} yields an isomorphism
\[D^2A\simeq A\oplus \tau_{\ge 2} DA^*\oplus D(\tau_{\ge 2} DA)\]
hence a split exact triangle
\[\sA_X \by{\epsilon_{\sA_X}}D^2\sA_X \to \tau_{\ge 2} DA^*\oplus D(\tau_{\ge 2} DA)\by{+1}.\]

Let now $M^0(X)$ be the reduced motive of $X$, sitting in the (split) exact triangle $M^0(X)\to M(X)\to \Z\by{+1}$, as in the proof of Lemma \ref{l3.2} c). The map $a_X$ induces a map $a_X^0:M^0(X)\to A$, hence a dual map
\[D(a_X^0):A^*\oplus \tau_{\ge 2} DA\simeq DA\to DM^0(X)\simeq \Pic_{X/F}\]
where the left (resp. right) hand isomorphism follows from Lemma \ref{l6.4} (resp. from \eqref{eq3.4}). By construction, $D(a_X^0)$ restricts to the isomorphism $A^*\iso \Pic^0_{X/F}$. Dualising the resulting exact triangle $A^*\to D M^0(X)\to \NS_X\by{+1}$ and reusing  Lemma \ref{l6.4}, we get an exact triangle
\[\NS_X^*[1]\to D^2M^0(X)\to A\oplus \tau_{\ge 2} D A^*\by{+1}\]
where $\NS_X^*$ is the Cartier dual of $\NS_X$. It follows that 
\[H^0(k,D^2M^0(X))=A(k)\] and therefore that $H^0(k, D^2M(X))=\sA_X(k)$, the map induced by $D^2(a_X)$ being the canonical injection. We thus get the requested diagram, with $Q=H^0(k,D(\tau_{\ge 2} DA))$.
%
%
%
\end{proof}

\section{Proof of Theorem \ref{t6.1}}

Instead of Lewis' idea to use the complex Abel-Jacobi map, we use the $l$-adic Abel-Jacobi map in order to cover the case of arbitrary characteristic.

 We may find a regular $\Z$-algebra $R$ of finite type, a homomorphism $R\to F$,  and a smooth projective scheme $p:\sX\to \Spec R$, such that $X=\sX\otimes_R F$. By a direct limit argument, it suffices to show the theorem when $F$ is the algebraic closure of the quotient field of $R$ and, moreover, to show that the composition
\[H^1(\sX,\sK_2)\to H^1_\ind(X,\sK_2)\by{\bar \delta} \Hom(\Griff_1(X),F^*)\]
has image killed by $e$.

Let $l$ be a prime number different from $\car F$. We may assume that $l$ is invertible in $R$. We have $l$-adic regulator maps
\[H^1(\sX,\sK_2)\by{c} H^3_\et(\sX,\Z_l(2)),\quad H^{d-1}(\sX,\sK_{d-1})\by{c'} H^{2d-2}_\et(\sX,\Z_l(d-1))\]
and two compatible pairings
\begin{multline}\label{eq6.2}
H^1(\sX,\sK_2)\times H^{d-1}(\sX,\sK_{d-1})\to H^{d}(\sX,\sK_{d+1})\\
\by{p_*} H^0(R,\sK_1)= R^*
\end{multline}
\begin{multline}\label{eq6.3}
H^3_\et(\sX,\Z_l(2))\times H^{2d-2}_\et(\sX,\Z_l(d-1))\to H^{2d+1}_\et(\sX,\Z_l(d+1))\\
\by{p_*} H^1_\et(R,\Z_l(1)).
\end{multline}


The Leray spectral sequence for the projection $p$ yields a filtration $F^rH^*_\et(\sX,\Z_l(\bullet))$ on the $l$-adic cohomology of $\sX$.

Let $H^{d-1}(\sX,\sK_{d-1})_0={c'}^{-1}(F^1H^{2d-2}_\et(\sX,\Z_l(d-1)))$ and $H^1(\sX,\sK_2)_0=c^{-1}(F^1H^3_\et(\sX,\Z_l(2)))$.

\begin{lemma}\label{l6.2} The restriction of \eqref{eq6.2} to $H^1(\sX,\sK_2)_0\times H^{d-1}(\sX,\sK_{d-1})_0$ has image in $R^*\{l'\}$, the subgroup of $R^*$ of torsion prime to $l$.
\end{lemma}

\begin{proof} Since $R$ is a finitely generated $\Z$-algebra, its group of units $R^*$ is a finitely generated $\Z$-module, hence the map $R^*\otimes\Z_l\to H^1_\et(R,\Z_l(1))$ from Kummer theory is injective; therefore the induced map $R^*\to H^1_\et(R,\Z_l(1))$ has finite kernel of cardinality prime to $l$. It therefore suffices to show that the restriction of \eqref{eq6.3} to 
\[F^1H^3_\et(\sX,\Z_l(2))\times F^1H^{2d-2}_\et(\sX,\Z_l(d-1))\]
is $0$. By multiplicativity of the Leray spectral sequences, it suffices to show that $p_*(F^2 H^{2d+1}_\et(\sX,\Z_l(d+1)))=0$.

Since $\dim X=d$, we have $H^0_\et(R,H^{2d+1}_\et(X,\Z_l(d+1))=0$ and hence $H^{2d+1}_\et(\sX,\Z_l(d+1)))=F^1 H^{2d+1}_\et(\sX,\Z_l(d+1)))$. The edge map
\[F^1 H^{2d+1}_\et(\sX,\Z_l(d+1)))\to H^1_\et(R,H^{2d}_\et(X,\Z_l(d+1)))\]
coincides with the map $p_*$ of \eqref{eq6.3} via the isomorphism
\[H^{2d}_\et(X,\Z_l(d+1))\by{p_*} H^0_\et(F,\Z_l(1))=\Z_l(1).\]

This concludes the proof.
\end{proof}

Passing to the $\colim$ in Lemma \ref{l6.2}, we find that the pairing
\[H^1(X,\sK_2)_0\times CH^{d-1}(X)_0\to F^*\]
has image in $F^*\{l'\}$.

\begin{lemma}\label{l6.1} The group $H^1(X,\sK_2)/H^1(X,\sK_2)_0$ is finite of exponent dividing $e_l$.
\end{lemma}

\begin{proof} It suffices to observe that the regulator map
\[H^1(X,\sK_2)\to H^3_\et(X,\Z_l(2))\]
has finite image \cite[Th. 2.2]{ctr}.
\end{proof}

Lemmas \ref{l6.2} and \ref{l6.1} show that the pairing $H^1(X,\sK_2)\times CH^{d-1}(X)\to F^*$ has image in a group of  roots of unity whose $l$-primary component is finite of exponent $e_l$ for all primes $l\ne \car F$. This completes the proof of Theorem \ref{t6.1}.

\section{Questions and remarks}

(1) Does the conclusion of Proposition \ref{p3.2} remain true when $\dim X>2$?

(2) Can one give an \emph{a priori}, concrete, description of the extension in Theorem \ref{t2.1} (iii)?

(3) It is known that $\Griff_1(X)\otimes \Q$ (\emph{resp} $\Griff_1(X)/l$ for some primes $l$) may be nonzero for some $3$-folds $X$ \cite{griffiths,bl-es}; these groups may not even be finite dimensional, \emph{e.g.} \cite{clemens,schoen}. Can one find examples for  which $\Hom(\Griff_1(X),\Z)\ne 0$? 

(4) To put the previous question in a wider context, let $A$ be a torsion-free abelian group. Replacing $\bG_m$ by $\sF=\bG_m\otimes A$ in Theorem \ref{t2.1} yields the following computation (with same proofs):
\begin{thlist}
\item $R^0_\nr\sF(X) =F^*\otimes A$.
\item $R^1_\nr\sF(X)\iso \Pic^\tau(X)\otimes A$.
\item There is a short exact sequence
\begin{equation}\label{eq6.1}
0\to D^1(X)\otimes A\to R^2_\nr\sF(X)\to \Hom(\Griff_1(X),A)\to 0.
\end{equation} 
\end{thlist}

Taking $A=\Q$ we get examples, from the nontriviality of $\Griff_1(X)\otimes \Q$, where $R^2_\nr\sF(X)$ is not reduced to $D^1(X)\otimes A$. But, choosing $X$ such that $\Griff_1(X)\otimes \Q$ is not finite dimensional and varying $A$ among $\Q$-vector spaces, \eqref{eq6.1} also shows that \emph{the functor $\sF\mapsto R^2_\nr \sF$ does not commute with infinite direct sums}. (Therefore $R_\nr$ cannot have a right adjoint.) This is all the more striking as $R^0_\nr$ does commute with infinite direct sums, which is clear from Formula \eqref{eq1} in the introduction.

We don't know whether $R^1_\nr$ commutes with infinite direct sums or not.

\end{document}